\documentclass[11pt]{amsart}

\usepackage{latexsym, amssymb,enumerate}
\usepackage[T1]{fontenc}
\usepackage{textcomp}
\usepackage{appendix}

\newcommand{\be}{\begin{equation}}
\newcommand{\ee}{\end{equation}}
\newcommand{\beq}{\begin{eqnarray}}
\newcommand{\eeq}{\end{eqnarray}}

\def\p{\partial}

\def\k{\kappa}

\def\<{\langle}
\def\>{\rangle}
\newtheorem{prop}{Proposition}[section]
\newtheorem{theo}[prop]{Theorem}

\newtheorem{coro}[prop]{Corollary}

\newtheorem{defi}[prop]{Definition}

\def\begeq{\begin{equation}}
\def\endeq{\end{equation}}
\def\p{\partial}

\def\tr{{\rm tr}}

\def\d{\delta}
\def\a{\alpha}

\def\l{\lambda}

\def\odot{\setbox0=\hbox{$\bigcirc$}\relax \mathbin {\hbox
to0pt{\raise.5pt\hbox to\wd0{\hfil $\wedge$\hfil}\hss}\box0 }}

\numberwithin{equation} {section}

\begin{document}
\author{Jie Wu}
\address{School of Mathematical Sciences, University of Science and Technology
of China Hefei 230026, P. R. China
\and
 Albert-Ludwigs-Universit\"at Freiburg,
Mathematisches Institut
Eckerstr. 1,
D-79104, Freiburg, Germany
}
\email{jie.wu@math.uni-freiburg.de}
\thanks{The first author is supported by SFB/TR71
``Geometric partial differential equations''  of DFG}
\author{Chao Xia}\address{Max-Planck-Institut f\"ur Mathematik in den Naturwissenschaften, Inselstr. 22, D-04103, Leipzig, Germany}\thanks{The second author is supported by funding from the European Research Council
under the European Union's Seventh Framework Programme (FP7/2007-2013) / ERC grant
agreement no. 267087.}

\email{chao.xia@mis.mpg.de}
\begin{abstract}{
In this paper, we study rigidity problems for hypersurfaces with constant curvature quotients $\frac{\mathcal{H}_{2k+1}}{\mathcal{H}_{2k}}$ in the warped product manifolds. Here $\mathcal{H}_{2k}$ is the $k$-th Gauss-Bonnet curvature and $\mathcal{H}_{2k+1}$ arises from the first variation of the total integration of $\mathcal{H}_{2k}$. Hence the quotients considered here are in general different from $\frac{\sigma_{2k+1}}{\sigma_{2k}}$, where $\sigma_k$ are the usual mean curvatures. We prove several rigidity and Bernstein type results for compact or non-compact hypersurfaces corresponding to such quotients. }
\end{abstract}


\title[Hypersurfaces with constant curvature quotients]{Hypersurfaces with constant curvature quotients in warped product manifolds}
\maketitle

\section{Introduction}

Let $\Sigma^{n-1}$ be a closed smooth hypersurface  isometrically immersed in an $n$-dimensional Riemannian manifold $(M^n, g)$. Assume that $\Sigma_t$ is a variation of $\Sigma$ with the unit normal vector field $\nu_t$ as the variational vector field.  It is well known that the first variation of the area functional
${\rm Area}(\Sigma_t)$ is given by
\begin{eqnarray*}
\frac{d}{dt}\bigg|_{t=0}{\rm Area}(\Sigma_t)=\int_{\Sigma} H d\mu,
\end{eqnarray*}
where $H$ is the mean curvature of $\Sigma$ with respect to the inner normal  and $d\mu$ is the area element of $\Sigma$.
On the other hand, it is well known that the first variation of the total scalar curvature functional $\int_{\Sigma}R d\mu$ is given by
\begin{eqnarray*}
\frac{d}{dt}\bigg|_{t=0}\int_{\Sigma_t}R d\mu_t=\int_\Sigma -2\sum_{i,j=1}^{n-1}E^{ij} h_{ij}d\mu,
\end{eqnarray*}
where $E^{ij}=R^{ij}-\frac12 Rg^{ij}$ and $h_{ij}$ are respectively the Einstein tensor and  the second fundamental form of $\Sigma$ with respect to the inner normal  in local coordinates.

There is a natural generalization of scalar curvature, called Gauss-Bonnet curvatures $L_k$ for an integer $1\leq k\leq \frac{n-1}{2}$ for $(n-1)$-dimensional Riemannian manifolds. $L_k$ are intrinsic curvature functions. The Pfaffian in Gauss-Bonnet-Chern formula is the highest order Gauss-Bonnet curvature.  $L_2$  appeared first in the paper of Lanczos \cite{Lan} in 1938 and has been intensively studied in the theory of Gauss-Bonnet gravity, which is a generalization of Einstein gravity.

The first variation of the total Gauss-Bonnet curvature functional $\int_{\Sigma} L_k d\mu$ has been considered long time ago by Lovelock \cite{Lovelock}. In \cite{LiAM} Li also computed the first  variation of these functionals as well as the second variation for submanifolds in general ambient Riemannian manifolds. Recently an alternative computation was given by Labbi in \cite{Labbi}. It tells that
\begin{eqnarray*}
\frac{d}{dt}\bigg|_{t=0}\int_{\Sigma_t}L_k d\mu_t=\int_\Sigma -2\sum_{i,j=1}^{n-1}E_{(k)}^{ij} h_{ij}d\mu,
\end{eqnarray*}
where $E_{(k)}^{ij}$ is the generalized Einstein tensor defined by \eqref{Einstein} in Section 2. Labbi \cite{Labbi2} referred to the critical point of $\int_{\Sigma} L_k d\mu$ as $2k$-minimal submanifolds. In this sense, the ordinary minimal submanifolds are referred as $0$-minimal submanifolds.

For the ambient space $M^n=\mathbb{R}^n$, by the Gauss equation, one can verify that $L_k=(2k)!\sigma_{2k}$ and $-2\sum_{i,j=1}^{n-1}E_{(k)}^{ij} h_{ij}=(2k+1)!\sigma_{2k+1}$, where $\sigma_k$ are the usual mean curvatures  defined by the elementary symmetric functions of principal curvatures of associated hypersurfaces. Hence  the Gauss-Bonnet curvatures $L_k$ as well as the integrand $-2\sum_{i,j=1}^{n-1}E_{(k)}^{ij} h_{ij}$ appear like higher order mean curvatures.

Throughout this paper, we use the notations $$\mathcal{H}_{2k}:=L_k,\quad \mathcal{H}_{2k+1}:=-2\sum_{i,j=1}^{n-1}E_{(k)}^{ij}h_{ij},$$ and call them $(2k)$-mean curvature and $(2k+1)$-mean curvature respectively.  We emphaysize here that in general these mean curvatures are different from the usual ones defined by $\sigma_k$ except $\mathcal{H}_0$ and $\mathcal{H}_1$. The $0$-mean curvature $\mathcal{H}_0$ is equal to $1$ and the $1$-mean curvature $\mathcal{H}_1$ is equal to the usual mean curvature $H$.

In this paper, we will consider some rigidity problem related to $\mathcal{H}_{2k}$ and $\mathcal{H}_{2k+1}$ in a class of Riemannian manifolds --- warped product manifolds.
A warped product manifold $(M, \bar{g})$ is the product manifold of  one dimensional interval and  an $(n-1)$-dimensional Riemannian manifold with some smooth positive warping function. Precisely, $$M= [0,\bar{r})\times_\l N^{n-1}, (0<\bar{r}\leq \infty)$$ is equipped with $$\bar{g}=dr^2+\lambda(r)^2g_N,$$  where $(N^{n-1}, g_N)$ is an $(n-1)$-dimensional Riemannian manifold and $\l:[0,\bar{r})\to\mathbb{R}_+$ is a smooth positive function.

The rigidity problems for  hypersurfaces in Riemannian manifolds with constant curvature functions are one of the central problems in  classical differential geometry. Historically, the rigidity problems for hypersurfaces in the Euclidean space  was studied by Liebmann \cite{Li}, Hsiung \cite{Hs}, S\"uss \cite{Su}, Alexandrov \cite{A}, Reilly \cite{Reilly}, Ros \cite{Ros0}, Korevaar \cite{Kor} etc.. Recently, Many works concerning about rigidity for hypersurfaces in warped product manifolds appeared, see for example  Montiel \cite{Mon}, Al\'ias-Impera-Rigoli \cite{AIR}, Brendle \cite{Brendle}, Brendle-Eichmair \cite{BE}, Wu-Xia \cite{WuX} and the references therein.

In all above works, the curvature functions are related to  the elementary symmetric functions $\sigma_k$ of principal curvatures of hypersurfaces.  Our concern in this paper is the curvature functions $\mathcal{H}_{2k}$ and $\mathcal{H}_{2k+1}$. 
In view of Gauss equation, for hypersurfaces in general ambient Riemannian manifolds, $\mathcal{H}_{2k}$ and $\mathcal{H}_{2k+1}$ depend not only on $\sigma_k$ but also  on the Riemannian curvature tensor of the ambient manifolds. Therefore, except for the case that the ambient spaces are the space forms, for which $\mathcal{H}_{2k}$ and $\mathcal{H}_{2k+1}$ can be written as linear combinations of $\sigma_k$, one cannot express them as pure functions on the principal curvatures of hypersurfaces.

The first attempt in which we succeed is the rigidity on the curvature quotients $\frac{\mathcal{H}_{2k+1}}{\mathcal{H}_{2k}}$ in a class of warped product manifolds. These quotients can be viewed as a generalization of the usual mean curvature $H$ since the case $k=0$ corresponds to $H$. We remark that the rigidity on the quotients of $\sigma_k$ in a class of warped product manifolds has been considered by the authors in \cite{WuX}. However, as mentioned before, these two kinds of quotients have large differences in general.
Many techniques seem to be difficult to apply for the quotients $\frac{\mathcal{H}_{2k+1}}{\mathcal{H}_{2k}}$ considered here.

The first main result of this paper is stated as follows.
\begin{theo}\label{mainthm}
Let $(M^n, \bar{g})$ be an $n$-dimensional warped product manifold $[0,\bar{r})\times_\l N^{n-1}$ whose warped product function satisfies  \begin{eqnarray}\label{logconvex}
\l\l''-(\l')^2\geq 0 \hbox{ (i.e. } \log \l\hbox{ is convex).}
\end{eqnarray}
Let $\Sigma^{n-1}$ be a closed star-shaped hypersurface in $M$
such that the generalized Einstein tensor $E_{(k)}$ being  semi-definite on $\Sigma$. For any integer $k$ with $0\leq k<\frac{n-1}{2}$ and $\mathcal{H}_{2k}$ not vanishing on $\Sigma$,
if the curvature quotient $\frac{\mathcal{H}_{2k+1}}{\mathcal{H}_{2k}}$ is a constant,
 then $\Sigma$ is a slice $\{r_0\}\times N$ for some $r_0\in [0,\bar{r})$ and the constant is $(n-1-2k)\log \l (r_0)$.\end{theo}

The star-shapedness means that $\Sigma$ can be written as a graph over $N$, alternatively, $\<\frac{\p}{\p r},\nu\>\geq 0,$ where $\nu$ is the outer normal of $\Sigma$.
The method to prove Theorem \ref{mainthm} is to apply the maximum principle to an elliptic equation. This method was previously indicated by Montiel \cite{Mon} and has been used widely by Al\'ias et. al.  \cite{AC},  \cite{AIR2}.

The condition \eqref{logconvex} imposed on $M$ only depends on the warped product function $\l$ but not on the fiber manifold $N$. We notice that the condition excludes the usual space forms $\mathbb{R}^n, \mathbb{S}_+^n$ and $\mathbb{H}^n$ in which cases $\l\l''-(\l')^2=-1$. For $\mathbb{R}^n$, since the quotient $\frac{\mathcal{H}_{2k+1}}{\mathcal{H}_{2k}}$ is equal to $\frac{\sigma_{2k+1}}{\sigma_{2k}}$, the result still holds, see Korevaar \cite{Kor} and Koh \cite{Koh2}. We will consider the case  $\mathbb{S}_+^n$ (semi-sphere) and $\mathbb{H}^n$ elsewhere since the proof has a different flavor.
We also notice that the condition \eqref{logconvex} is satisfied by some local space forms such as $[0,\infty)\times_{e^r} \mathbb{R}^{n-1}$ or $[0,\infty)\times_{\cosh r} \mathbb{R}^{n-1}$.  There are also non constant curvature manifolds which satisfy \eqref{logconvex}.
A typical example for which the condition \eqref{logconvex} is satisfied
is  the so-called Kottler-Schwarzschild spaces $[0,\infty)\times_\l N(\kappa)$, whose warped product fact $\l$ satisfies
$\l'(r)=\sqrt{\kappa+\l(r)^2-2m\l(r)^{2-n}}$ and $N(\kappa)$ is a closed space form of constant sectional curvature $\kappa=0$ or $-1$.  See Appendix A for a detailed explanation.

Note that $E_{(1)}^{ij}=R^{ij}-\frac12 Rg^{ij}$ is the Einstein tensor, so that in $k=1$ case,  the semi-definite condition of $E_{(1)}$ is just
 the
semi-definiteness of the Einstein tensor. In particular, if $M=\mathbb{R}^n$, one readily sees that $-E_{(k)}=\frac{(2k)!}{2}T_{2k},$ where $T_{2k}$ is the $2k$-Newton tensor associated to the hypersurface $\Sigma$, and the semi-definite condition of $E_{(k)}$ relates to $2k$-convexity.

In order to extend the above result to non-compact hypersurfaces, we need a generalization of the Omori-Yau maximum principle for the trace type semi-elliptic operators. The classical Omori-Yau maximum principle is initially stated for the Laplacian  $\Delta$.  A Riemannian manifold $M$ is said to satisfy the Omori-Yau maximum
principle if for any function $u\in C^2(\Sigma)$ with $\sup_{\Sigma}u<+\infty$, there exists a sequence $\{p_i\}_{i\in\mathbb{N}}\subset\Sigma$ such that  for each $i$, the following holds:
$$ (1) \quad u(p_i)>\sup_{\Sigma}u-\frac{1}{i},\quad (2)\quad |\nabla u|(p_i)< \frac 1 i, \quad(3)\quad \Delta u(p_i)<\frac 1 i.$$
This principle was first proved by Omori \cite{Omori} and later generalized by Yau \cite{Yau} under the condition  that the Ricci curvature is bounded from below.
It has proved to be very useful in the framework of non-compact manifolds and has attracted considerable extending works.
It was improved by Chen-Xin \cite{CX} and Ratto-Rigoli-Setti \cite{RRS} by assuming that the radial curvature decays slower than a certain decreasing function. Recently, the essence of the Omori-Yau maximum principle was captured by  Pigola, Rigoli and Setti (see Theorem $1.9$ in \cite{PRS2}) that the validity of Omori-Yau maximum principle is assured by the existence of some non-negative $C^2$ function satisfying some appropriate requirements, and thus may not necessarily depend on the curvature bounds. Also in the same paper, they discussed the generalizations for trace type differential operators (see Definition \ref{maximum}) which will be used in this paper. For a detailed discussion of sufficient condition to guarantee the  Omori-Yau maximum principle for  trace type differential operators to hold, see Al\'ias-Impera-Rigoli \cite{AIR} or Section 3 below.

We have the following rigidity result for non-compact hypersurfaces.
\begin{theo}\label{mainthm2}
Let $(M^n, \bar{g})$ be an $n$-dimensional warped product manifold $[0,\bar{r})\times_\l N^{n-1}$ whose warped product function satisfies  $
\l\l''-(\l')^2\geq 0 $ with equality only at isolated points.
Let $(\Sigma^{n-1},g)$ be a complete non-compact star-shaped hypersurface in $M$, which is contained in a slab $[r_1,r_2]\times N$, such that the generalized Einstein tensor $E_{(k)}$ being semi-definite on $\Sigma$.
Assume the Omori-Yau maximum principle holds for the trace type operator $\tr_g(-2E_{(k)}\nabla_g^2)$ on $\Sigma$. For any integer $k$ with $0\leq k<\frac{n-1}{2}$ and $\mathcal{H}_{2k}$ not vanishing on $\Sigma$,
if the curvature quotient $\frac{\mathcal{H}_{2k+1}}{\mathcal{H}_{2k}}$ is a constant, then $\Sigma$ is a slice $\{r_0\}\times N$ for some $r_0\in [0,\bar{r})$ and the constant is $(n-1-2k)\log \l (r_0)$.
\end{theo}

Motivated by  analogous Bernstein type result on the quotient of the usual mean curvatures \cite{AL}, we can establish  corresponding result in our case. More precisely, instead of assuming the curvature quotient $\frac{\mathcal{H}_{2k+1}}{\mathcal{H}_{2k}}$ being constant,  we can establish the rigidity result via assuming a natural comparison inequality between $\frac{\mathcal{H}_{2k+1}}{\mathcal{H}_{2k}}$ and its value on the slices.
\begin{theo}\label{mainthm3}
Let $(M^n, \bar{g})$ be an $n$-dimensional warped product manifold $[0,\bar{r})\times_\l N$. Let $(\Sigma^{n-1},g)$ be a complete, star-shaped hypersurface in $M$, which is contained in a slab $[r_1,r_2]\times N$, such that the generalized Einstein tensor $E_{(k)}$ being semi-definite on $\Sigma$.
Assume
 the Omori-Yau maximum principle holds for the trace type operator $tr(-2E_{(k)}\nabla_g^2)$ on $\Sigma$ and the Gauss-Bonnet curvature $\mathcal H_{2k}$ is bounded by two positive constants, i.e. $0<C_1\leq \mathcal H_{2k}\leq C_2$. If $$\frac{\mathcal{H}_{2k+1}}{\mathcal{H}_{2k}}\leq (n-1-2k)\frac{\l'(r)}{\l(r)},\quad
 \mbox{and} \quad |\nabla r|\leq \inf_{\Sigma}\left((n-1-2k)\frac{\l'(r)}{\l(r)}-\frac{\mathcal{H}_{2k+1}}{\mathcal{H}_{2k}}\right),$$
then the hypersurface $\Sigma$ is a slice $\{r_0\}\times M$ for some $r_0\in [0,\bar{r})$.
\end{theo}

We remark that we do not assume the log-convexity of the warped product function for Theorem \ref{mainthm3}.

\

\section{Preliminaries}

In this section, we first recall the work of Lovelock \cite{Lovelock} on the generalized Einstein tensors and Gauss-Bonnet curvatures.
Throughout this paper, we use the notations $R_{ijkl}$, $R_{ij}$  and  $R$ to indicate the Riemannian $4$-tensor, the Ricci tensor in local coordinates and the scalar curvature respectively.  We use the metric $g$ to lower or raise an index and adopt the Einstein summation convention: repeated upper and lower indices will automatically be summed unless otherwise noted.

For an $(n-1)$-dimensional Riemannian manifold $(M^{n-1},g)$, the Einstein tensor $E_{ij}=R_{ij}-\frac 12 R g_{ij}$ is very important in theoretical physics. It is a conversed quantity, i.e.,
\[\nabla_j E^j_{i}=0,\]
where $\nabla$ is the covariant derivative with respect to the metric $g$.

In \cite{Lovelock} Lovelock studied the classification of tensors $A$ satisfying
\begin{itemize}
 \item [(i)]$A^{ij}=A^{ji}$, i.e, $A$ is symmetric.
\item[(ii)] $A^{ij}=A^{ij}(g,\partial g, \partial^2 g)$.
\item[(iii)] $\nabla_j A^{ij}=0$, i.e. $A$ is divergence-free.
\item [(iv)] $A^{ij}$ is linear in the second derivatives of
$g$.
\end{itemize}
It is clear that the Einstein tensor $E_{ij}$ satisfies all conditions.
Lovelock classified all 2-tensors satisfying (i)--(iii).
For an integer $0\leq k\leq \frac{n-1}{2}$, let us define a 2-tensor $E_{(k)}$ locally by
\beq\label{Einstein}
{E_{(k)}^{ij}}:=-\frac{1}{2^{k+1}}g^{l j} \d^{i  i_1i_2\cdots i_{2k-1}i_{2k}}
_{l j_1j_2\cdots j_{2k-1}j_{2k}}{R_{i_1i_2}}^{j_1j_2}\cdots
{R_{i_{2k-1}i_{2k}}}^{j_{2k-1}j_{2k}}.
\eeq
Here the generalized Kronecker delta is defined by
\begin{equation*}\label{generaldelta}
 \d^{j_1j_2 \cdots j_r}_{i_1i_2 \cdots i_r}=\det\left(
\begin{array}{cccc}
\d^{j_1}_{i_1} & \d^{j_2}_{i_1} &\cdots &  \d^{j_r}_{i_1}\\
\d^{j_1}_{i_2} & \d^{j_2}_{i_2} &\cdots &  \d^{j_r}_{i_2}\\
\vdots & \vdots & \vdots & \vdots \\
\d^{j_1}_{i_r} & \d^{j_2}_{i_r} &\cdots &  \d^{j_r}_{i_r}
\end{array}
\right).
\end{equation*}
One can check that $E_{(k)}$ satisfies (i)--(iii).
Lovelock proved that any 2-tensor satisfying (i)--(iii) has the form
\[\sum_{k}\a^k E_{(k)},\]
with certain constants $\a^k$, $k\ge 0$.
The $E_{(k)}$'s are called the generalized Einstein tensors.

For an integer $0\leq k\leq \frac{n-1}{2}$, the Gauss-Bonnet curvatures $L_k$ are defined by
\beq\label{Lk}
L_k:=\frac{1}{2^k}\d^{i_1i_2\cdots i_{2k-1}i_{2k}}
_{j_1j_2\cdots j_{2k-1}j_{2k}}{R_{i_1i_2}}^{j_1j_2}\cdots
{R_{i_{2k-1}i_{2k}}}^{j_{2k-1}j_{2k}}.
\eeq
When $2k=n-1$, $L_k$ is
the Euler density. When
$k<\frac{n-1}{2}$, $L_k$ is called the dimensional continued Euler density
in physics. We set $E_{(0)}=-\frac12 g$ and $L_0=1$.
It is clear from the definitions \eqref{Einstein} and \eqref{Lk} that
\beq\label{trEk}
 tr_g(E_{(k)}):=E_{(k)}^{ij}g_{ij}=-\frac{n-1-2k}{2}L_k.\eeq
 It is easy to see that $(E_{(1)})_{ij}=R_{ij}-\frac 12Rg_{ij}$  is the Einstein tensor and
$L_1=R$
 is  the scalar curvature.
One can also check that
\[
  E_{(2)}^{ij}=2RR^{ij}-4R^{is} {R_s}^{j} -4 R_{sl}
R^{silj}+2{R^i}_{klm}R^{jklm}-\frac12 g^{ij}L_2,
\]
and \[
L_2=\frac 14 \delta
^{i_1i_2i_3i_4}_{j_1j_2j_3j_4}{R^{j_1j_2}}_{i_1i_2}
{R^{j_3j_4}}_{i_3i_4}
=R_{ijsl}R^{ijsl}-4R_{ij}R^{ij}+R^2.\]

In \cite{Lovelock} Lovelock  proved that the first variational formula for the total Gauss-Bonnet curvature functional is given in terms of the generalized Einstein tensor. It was also presented in \cite{LiAM, Labbi}, although with different notation and formalism. For the convenience of readers, we include a proof here.
\begin{prop}[\cite{Lovelock}]
Let $(\Sigma^{n-1},g)$ be a closed manifold. Assume that $(\Sigma_t,g_t)$ is a variation of $\Sigma$ with $\frac{\partial}{\partial t}\big|_{t=0} g_{ij}=v_{ij}$ for a symmetric 2-tensor $v$, then
\begin{eqnarray}\label{1st}
\frac{d}{dt}\bigg|_{t=0}\int_{\Sigma_t}L_k d\mu_t=\int_\Sigma -E_{(k)}^{ij} v_{ij}d\mu.
\end{eqnarray}
In particular, if $(\Sigma^{n-1},g)$ is a closed, smooth hypersurface immersed in an $n$-dimensional Riemannian manifold $(M^n, \bar g)$ and the variational vector field is given by the outward unit normal $\nu$, then
\begin{eqnarray}\label{2st}
\frac{d}{dt}\bigg|_{t=0}\int_{\Sigma_t}L_k d\mu_t=\int_\Sigma -2E_{(k)}^{ij} h_{ij}d\mu.
\end{eqnarray}
where $h_{ij}$ denotes the second fundamental form of $\Sigma$ with respect to $-\nu$.
\end{prop}
\begin{proof}
By the simple fact that $\frac{d}{dt}\bigg|_{t=0} d\mu_t=\frac 12\tr_g v d\mu$ and the definition of $L_k$, we compute
\begin{eqnarray}
\frac{d}{dt}\bigg|_{t=0}\int_{\Sigma_t}L_k d\mu_t&=&\int_{\Sigma}\frac{d}{dt}\bigg|_{t=0}L_kd\mu+\int_{\Sigma}\frac 12L_k \tr_g vd\mu\nonumber\\
&=&\int_{\Sigma}k{P_{(k)}^{ij}}_{sl}\frac{d}{dt}\bigg|_{t=0}{R_{ij}}^{sl}d\mu+\int_{\Sigma}\frac 12L_k\tr_g vd\mu,\label{1}
\end{eqnarray}
where the $4$-tensor $P_{(k)}$ is given by
\begin{equation}\label{Pk}
P_{(k)}^{st lj}:=\frac{1}{2^k}\d^{i_1i_2\cdots i_{2k-3}i_{2k-2}st}
_{j_1j_2\cdots  j_{2k-3}j_{2k-2} j_{2k-1}j_{2k}}{R_{i_1i_2}}^{j_1j_2}\cdots
{R_{i_{2k-3}i_{2k-2}}}^{j_{2k-3}j_{2k-2}}g^{j_{2k-1}l}g^{j_{2k}j},
\end{equation}
and
$${P_{(k)}^{ij}}_{sl}=P_{(k)}^{ijpq}g_{sp}g_{lq}.$$
We remark that $P_{(k)}$ shares the same symmetry as the Riemann curvature tensor, that is
\begin{equation}\label{symmetry}
P_{(k)}^{stjl}=-P_{(k)}^{tsjl}=-P_{(k)}^{stlj}=P_{(k)}^{jlst}.
\end{equation}
Furthermore, by applying the second Bianchi identity of the curvature tensor, one can check that $P_{(k)}$ has the crucial property of being divergence-free (see Lemma 2.2 in \cite{GWW} for a proof)
\begin{equation}\label{div-free}
\nabla_s P_{(k)}^{stjl}=0.
\end{equation}
To calculate the first term in (\ref{1}), we recall that if $\frac{\partial}{\partial t} g=v,$ then the evolution equation of the curvature tensor is given by
(cf. (2.66) in \cite{CLN})
$$\frac{d}{dt}R_{ijsl}=-\frac{1}{2}(\nabla_i\nabla_j v_{sl}-\nabla_i\nabla_l v_{sj}-\nabla_s\nabla_j v_{il}+\nabla_s\nabla_lv_{ij}-R_{ijsm}{v^m}_l-R_{ijml}{v^m}_s).$$
Then we use (\ref{symmetry}) and (\ref{div-free}) to compute that
\begin{eqnarray}
&&\int_{\Sigma}k{P_{(k)}^{ij}}_{sl}\left(\frac{d}{dt}\bigg|_{t=0}{R_{ij}}^{sl}\right)d\mu\nonumber\\
&=&\int_{\Sigma}k{P_{(k)}^{ij}}_{sl}\bigg(\frac 12(-\nabla_i\nabla_j v^{sl}\!+\!\nabla_i\nabla^l v^s_j\!+\!\nabla^s\nabla_j v_{i}^l\!-\!\nabla^s\nabla^l v_{ij})\!+\nonumber\\
&&\quad+\!\frac 12({R_{ijm}}^lv^{ms}\!-\!{R_{ijm}}^sv^{ml})+(-{R_{ijp}}^l v^{sp}-R_{ij\;q}^{\;\;\;s} v^{lq})\bigg)d\mu\nonumber\\
&=&-\int_{\Sigma}k{P_{(k)}^{ij}}_{sl}{R_{ijm}}^lv^{sm}d\mu,\label{2}
\end{eqnarray}
where in
 the last equality we  used (\ref{div-free}), (\ref{symmetry}) and the simple observation that $({R_{ijm}}^lv^{ms}\!-\!{R_{ijm}}^sv^{ml})$ and $(-{R_{ijp}}^l v^{sp}-R_{ij\;q}^{\;\;\;s} v^{lq})$ are both  anti-symmetric with respect to the pair $(s,l)$.

Going back to (\ref{1}), we obtain that
$$\frac{d}{dt}\bigg|_{t=0}\int_{\Sigma_t}L_k d\mu_t=\int_{\Sigma}\left(-k{P_{(k)}^{ijs}}_{l}{R_{ij}}^{ml}+\frac 12L_kg^{ms}\right)v_{ms}d\mu.$$
On the other hand, from the definitions (\ref{Einstein}), (\ref{Lk}) and (\ref{Pk}), it is direct to check that
\begin{eqnarray*}
{E_{(k)}^{ms}}&=&-\frac{1}{2^{k+1}}g^{l s} \d^{m  i_1i_2\cdots i_{2k-1}i_{2k}}
_{l j_1j_2\cdots j_{2k-1}j_{2k}}{R_{i_1i_2}}^{j_1j_2}\cdots
{R_{i_{2k-1}i_{2k}}}^{j_{2k-1}j_{2k}}\\
&=&-\frac{1}{2^{k+1}}g^{m s} \d^{m  i_1i_2\cdots i_{2k-1}i_{2k}}
_{m j_1j_2\cdots j_{2k-1}j_{2k}}{R_{i_1i_2}}^{j_1j_2}\cdots
{R_{i_{2k-1}i_{2k}}}^{j_{2k-1}j_{2k}}\\
&&-\frac{2k}{2^{k+1}}g^{i_1 s} \d^{m  i_1i_2\cdots i_{2k-1}i_{2k}}
_{i_1 j_1j_2\cdots j_{2k-1}j_{2k}}{R_{i_1i_2}}^{j_1j_2}\cdots
{R_{i_{2k-1}i_{2k}}}^{j_{2k-1}j_{2k}}\\
&=&-\frac 1 2L_kg^{ms}+k{P_{(k)}^{ijs}}_{l}{R_{ij}}^{ml}.
\end{eqnarray*}
Hence we complete the proof of (\ref{1st}).

In the case that $\Sigma$ is a hypersurface, one only needs to note that $\frac{\partial}{\partial t}g_{ij}=2h_{ij}$ for the evolving hypersurfaces. \end{proof}

The second aim of this section is to give several simple facts on the warped product manifolds. Let $M^n= [0,\bar{r})\times_\l N^{n-1}$ $(0<\bar{r}\leq \infty)$ be a warped product manifold equipped with a Riemannian metric  $$\bar{g}=dr^2+\lambda(r)^2g_N.$$  
where $\l:[0,\bar{r})\to\mathbb{R}$ is a smooth positive function.
Let $\Sigma$ be a smooth hypersurface in $(M, \bar{g})$ with induced metric $g$. We denote by $\bar{\nabla}$ and $\nabla$ the covariant derivatives with respect to $\bar g$ and $g$ respectively.
We define a vector field $X$ on $M$ by $$X(r)=\l(r)\frac{\p}{\p r}.$$
Let $\{e_1,\cdots,e_{n-1}\}$ be a local frame on $\Sigma$, it is well known that $X$ is a conformal Killing vector field
satisfying
\beq\label{eq001}
\bar\nabla_{e_i} X(r)=\l'(r)e_i.
\eeq

We denote by $r$ the height function which is obtained by the projection of $\Sigma$ in $M$ onto the first factor $[0,\bar{r})$. Let $\phi(r)$ be a primitive of $\l(r)$.

\begin{prop}
The restriction of $\phi$ on $\Sigma$, still denoted by $\phi$, satisfies
\begin{eqnarray}\label{eqq1}
\nabla_{i}\nabla_j \phi(r)= \l'(r)g_{ij}-\langle X,\nu\rangle h_{ij}.
\end{eqnarray}
The height function $r$ on $\Sigma$ satisfies
\begin{eqnarray}\label{eqq2}
\nabla_{i}\nabla_j r= \frac{\l'(r)}{\l(r)}g_{ij}- \frac{\l'(r)}{\l(r)} \nabla_i r \nabla_j r-\langle \p_r,\nu\rangle h_{ij}.
\end{eqnarray}
Consequently, we have
\begin{eqnarray}\label{eqq3}
-2E_{(k)}^{ij}\nabla_i\nabla_j \phi(r)= (n-1-2k)\l'(r)\mathcal{H}_{2k}-\langle X,\nu\rangle \mathcal{H}_{2k+1}.
\end{eqnarray}
\begin{eqnarray}\label{eqq4}
-2E_{(k)}^{ij}\nabla_i\nabla_j r= (n-1-2k)\frac{\l'(r)}{\l(r)}\mathcal{H}_{2k}+ \frac{2\l'(r)}{\l(r)} E_{(k)}^{ij}\nabla_i r \nabla_j r-\langle \p_r,\nu\rangle \mathcal{H}_{2k+1}.
\end{eqnarray}
 \end{prop}
\begin{proof}
Using \eqref{eq001}, we have
\begin{eqnarray*}
\nabla_{i}\nabla_j \phi(r)&=&\bar{\nabla}_i \bar{\nabla}_j \phi- \langle \bar{\nabla} \phi(r),\nu\rangle h_{ij}\\
&=&\bar{\nabla}_i X_j- \langle X,\nu\rangle h_{ij}\\&=&\l'(r)g_{ij}-\langle X,\nu\rangle h_{ij}.
\end{eqnarray*}
Equation \eqref{eqq2} follows from \eqref{eqq1} and
\begin{eqnarray*}
\nabla_{i}\nabla_j r=\nabla_i \left(\frac{1}{\l(r)}\nabla_j \phi(r)\right)=\frac{1}{\l(r)} \nabla_i \nabla_j \phi(r)-\frac{\l'(r)}{\l(r)}\nabla_i r \nabla_j r.
\end{eqnarray*}
For equations \eqref{eqq3} and \eqref{eqq4}, we only need to notice that
$$-2E_{(k)}^{ij}g_{ij}=(n-1-2k)L_k=(n-1-2k)\mathcal{H}_{2k}.$$  and $$-2E_{(k)}^{ij}h_{ij}=\mathcal{H}_{2k+1}.$$
\end{proof}

\section{rigidity for the quotient  $\frac{\mathcal{H}_{2k+1}}{\mathcal{H}_{2k}}$}

In this section, we prove our main theorems.

\noindent\textit{Proof of Theorem \ref{mainthm}:}
Since $\Sigma$ is compact, there exist points $p_{min}, p_{max}\in\Sigma$ such that the height function $r$ attains its maximum and minimum values respectively, i.e.,
$$\min_{\Sigma} r=r(p_{min}),\quad \max_{\Sigma} r=r(p_{max}).$$
At these points,
\begin{eqnarray}\label{eqqq1}
\nabla r(p_{min})=\nabla r(p_{max})=0,
\end{eqnarray}
\begin{eqnarray}\label{eqqq2}
\nabla^2 r(p_{min})\geq 0, \quad \nabla^2 r(p_{max})\leq 0.
\end{eqnarray}
It follows from \eqref{eqqq1} and the star-shapedness of $\Sigma$ that
\begin{eqnarray}\label{eqqq3}
\langle\partial_r,\nu\rangle(p_{min})=\langle \p_r,\nu\rangle(p_{max})=1.
\end{eqnarray}
By using \eqref{eqqq1} and \eqref{eqqq3} in (\ref{eqq4}), we obtain
\begin{eqnarray}\label{eqqq5}
-2E_{(k)}^{ij}\nabla_i\nabla_j r(p_{min})&=&(n-1-2k)(\log \l)'(\min_{\Sigma} r)\mathcal{H}_{2k}(p_{min})-\mathcal{H}_{2k+1}(p_{min}),
\end{eqnarray}
\begin{eqnarray}\label{eqqq6}
-2E_{(k)}^{ij}\nabla_i\nabla_j r(p_{max})&=&(n-1-2k)(\log \l)'(\max_{\Sigma} r)\mathcal{H}_{2k}(p_{max})-\mathcal{H}_{2k+1}(p_{max}).
\end{eqnarray}

We claim that
 the quotient $\frac{\mathcal{H}_{2k+1}}{\mathcal{H}_{2k}}$ satisfies
\begin{eqnarray}\label{claim}
\\
\min_{\Sigma}\left(\frac{\mathcal{H}_{2k+1}}{\mathcal{H}_{2k}}\right)\leq (n-1-2k)(\log \l)'(\min_\Sigma r)\quad\mbox{and}\quad (n-1-2k)(\log \l)'(\max_{\Sigma} r)\leq \max_{\Sigma}\left(\frac{\mathcal{H}_{2k+1}}{\mathcal{H}_{2k}}\right).\nonumber
\end{eqnarray}

Consider first the case that $-2E_{(k)}^{ij}$ is positive semi-definite. It follows from \eqref{eqqq2}, \eqref{eqqq5} and \eqref{eqqq6} that
\begin{eqnarray}\label{eqqq7}
(n-1-2k)(\log \l)'(\min_{\Sigma} r)\mathcal{H}_{2k}(p_{min})-\mathcal{H}_{2k+1}(p_{min})\geq 0,
\end{eqnarray}
\begin{eqnarray}\label{eqqq8}
(n-1-2k)(\log \l)'(\max_{\Sigma} r)\mathcal{H}_{2k}(p_{max})-\mathcal{H}_{2k+1}(p_{max})\leq 0.
\end{eqnarray}
From the fact that $$-2E_{(k)}^{ij}g_{ij}=(n-1-2k)\mathcal{H}_{2k},$$ together with the assumption that $\mathcal{H}_{2k}$ is non-vanishing on $\Sigma$, we know that $\mathcal{H}_{2k}>0$.
Hence the claim in this case follows from  \eqref{eqqq7} and \eqref{eqqq8}  immediately.
For the second case that $-2E_{(k)}^{ij}$ is negative semi-definite,  similar argument applies by taking $\mathcal{H}_{2k}<0$ into account. We finish the proof of the claim.

Now using the assumption that $\log \l$ is convex, we obtain from \eqref{claim} that
\begin{eqnarray*}
\min_{\Sigma}\left(\frac{\mathcal{H}_{2k+1}}{\mathcal{H}_{2k}}\right)\leq (n-1-2k)(\log \l)'(\min_\Sigma r)\leq (n-1-2k)(\log \l)'(\max_{\Sigma} r)\leq \max_{\Sigma}\left(\frac{\mathcal{H}_{2k+1}}{\mathcal{H}_{2k}}\right).
\end{eqnarray*}
Since the quotient $\frac{\mathcal{H}_{2k+1}}{\mathcal{H}_{2k}}$ is constant, we have from above that
 \begin{eqnarray}\label{eqqq10}
\frac{\mathcal{H}_{2k+1}}{\mathcal{H}_{2k}}=(n-1-2k)(\log \l)'(\min_\Sigma r)=(n-1-2k)(\log \l)'(\max_\Sigma r),
\end{eqnarray}
which yields that $(\log \l)'(r)$ is a constant function on $\Sigma$.
Substituting \eqref{eqqq10} into (\ref{eqq3}), we have
\begin{eqnarray}\label{elliptic}
-2E_{(k)}^{ij}\nabla_i\nabla_j \phi(r)=\l(1-\langle \partial_r, \nu\rangle)\mathcal{H}_{2k+1}.
\end{eqnarray}

Notice that $\langle \partial_r, \nu\rangle\leq 1$ and $\mathcal{H}_{2k+1}=c\mathcal{H}_{2k}$ does not change sign on $\Sigma$. Applying the classical maximum principle to the elliptic equation \eqref{elliptic}, we conclude that $\phi(r)$ is a constant function on $\Sigma$. Since $\phi$ is an increasing function with respect to $r$ due to the fact $\phi'=\l>0$, we conclude that the height function $r$ is a constant function on $\Sigma$, i.e. $\Sigma$ is a slice $\{r_0\}\times N$.
\qed

\vspace{2mm}
To extend the previous result to  non-compact hypersurfaces, we will apply a generalization of the Omori-Yau maximum principle for trace type differential operators. Consider a Riemannian manifold $\Sigma$ and a semi-elliptic operator $L=\tr(T\circ \nabla^2)$, where $T:T\Sigma\rightarrow T\Sigma$ is a positive  semi-definite symmetric tensor, $\nabla^2$ is the Hessian on $\Sigma$ and $\tr$ is the trace operator with respect to the induced metric on $\Sigma$.
\begin{defi}\label{maximum}
 We say that the Omori-Yau maximum principle holds on $\Sigma$ for $L$, if for any function $u\in C^2(\Sigma)$ with $\sup_{\Sigma}u<+\infty$, there exists a sequence $\{p_i\}_{i\in\mathbb{N}}\subset\Sigma$ such that  for each $i$, the following holds:
$$(1) u(p_i)>\sup_{\Sigma}u-\frac{1}{i},\quad (2) |\nabla u|(p_i)< \frac 1 i, \quad(3) Lu(p_i)<\frac 1 i.$$
Since $\inf_{\Sigma}u=-\sup_{\Sigma}(-u), $ the above is equivalent to that for any function $u\in C^2(\Sigma)$ with $\inf_{\Sigma}u>-\infty$, there exists a sequence $\{p_i\}_{i\in\mathbb{N}}\subset\Sigma$ such that  for each $i$, the following holds:
$$(1) u(p_i)<\inf_{\Sigma}u+\frac{1}{i}, \quad(2) |\nabla u|(p_i)< \frac 1 i, \quad(3) Lu(p_i)>-\frac 1 i.$$
\end{defi}

\vspace{2mm}
Assume the generalized Omori-Yau maximum principle holds for trace-type operator $L=\tr(-2E_{(k)}\nabla^2)$, one can prove the analogous result for  non-compact hypersurfaces.

\vspace{2mm}
\noindent\textit{Proof of Theorem \ref{mainthm2}:}
By the generalized Omori-Yau maximum principle, we have two sequences $\{p_i\}$ and $\{q_i\}$ with properties
\begin{eqnarray*}&(i)&\lim_{i\rightarrow+\infty}\phi(r(p_i))=\sup_{\Sigma}\phi(r),\; \lim_{i\rightarrow+\infty}\phi(r(q_i))=\inf_{\Sigma}\phi(r);\\
&(ii)&|\nabla\phi(r)|(p_i)=\l(r(p_i))|\nabla r|(p_i)<\frac 1 i,\;|\nabla\phi(r)|(q_i)=\l(r(p_i))|\nabla r|(q_i)<\frac 1 i;\\
&(iii)&\tr\left(-2E_{(k)}\nabla^2 \phi(r)\right)(p_i)<\frac 1 i,\; \tr\left(-2E_{(k)}\nabla^2 \phi(r)\right)(q_i)>-\frac 1 i.\end{eqnarray*}

Since $\phi(r)$ is strictly increasing due to $\phi'(r)=f(r)>0$, we have $$\lim_{i\rightarrow +\infty}r(p_i)=\sup_{\Sigma}r,\quad \lim_{i\rightarrow\infty}r(q_i)=\inf_{\Sigma}r,$$
and thus
$$\lim_{i\rightarrow +\infty}\langle \p_r,\nu\rangle(p_i)=\lim_{i\rightarrow +\infty}\langle \p_r,\nu\rangle(q_i)=1.$$
Using the above facts in (\ref{eqq3}) and letting $i\rightarrow +\infty$, we get
\begin{eqnarray}\label{eqqqq1}
(n-1-2k)(\log \l)'(\sup_{\Sigma}r)\leq \frac{\mathcal{H}_{2k+1}}{\mathcal{H}_{2k}}\leq (n-1-2k)(\log \l)'(\inf_{\Sigma}r).
\end{eqnarray}
By the assumption that $(\log \l)''\geq 0$ with equality only at isolated points,
we obtain the desired result that $r$ is constant. That is, $\Sigma$ is a slice $\{r_0\}\times M$.
\qed

\vspace{2mm}
In the following, we discuss some sufficient condition to guarantee the generalized Omori-Yau maximum principle to hold.
Inspired by Pigola-Rigoli-Setti \cite{PRS2}, Al\'ias, Impera and Rigoli (see \cite{AIR}, Theorem 1 and Corollary 3)  proved that the Omori-Yau maximum principle holds for a trace type elliptic operator $L=\tr(T\circ\nabla^2)$ with positive semi-definite $T$ satisfying $\sup_{\Sigma} \tr T<\infty$ on a Riemannian manifold $\Sigma$, provided that the radial sectional curvature  (the sectional curvature of the $2$-planes containing $\nabla \rho$, where $\rho$ is the distance function on $\Sigma$ from a fixed point in $\Sigma$) of $\Sigma$ satisfies the condition \begin{eqnarray}\label{decaycond}
K^{\mbox{rad}}_{\Sigma}(\nabla \rho, \nabla \rho) > -G(\rho),
\end{eqnarray}
where $G : [0,+\infty)\rightarrow \mathbb{R}$ is
a smooth function satisfying
\begin{eqnarray}\label{condition}
G(0) > 0,\; G(t) > 0,\; \int_{0}^{+\infty} \frac{1}{\sqrt{G(t)}}=+\infty \quad\mbox{ and}\quad
\limsup_{t\rightarrow+\infty}\frac{tG(\sqrt t)}{G(t)}<+\infty.
\end{eqnarray}
A special case for which (\ref{decaycond}) holds is that the sectional curvature of $\Sigma$ is bounded from below (one can choose $G(\rho)=\rho^2$).

In particular, Al\'ias, Impera and Rigoli proved (see \cite{AIR}, Corollary 4)  that for a hypersurface $\Sigma$ in a slab of a warped product manifold  $[r_1,r_2]\times N$,    (\ref{decaycond}) holds for $L$ with positive semi-definite $T$ satisfying $\sup_{\Sigma} \tr T<\infty$, provided that the radial sectional curvature of the fiber manifold $N$ satisfies
\begin{eqnarray}\label{decaycond1}
K^{\mbox{rad}}_{N}(\nabla^N \hat\rho, \nabla^N \hat\rho) > -G(\hat\rho),
\end{eqnarray}
where $\hat\rho$ is  the distance function on the fiber $N$ from a fixed point in $N$ and $G : [0,+\infty)\rightarrow \mathbb{R}$ is
a smooth function  satisfying the conditions listed in (\ref{condition}), together with $\sup_{\Sigma}\|h\|^2<+\infty$ on $\Sigma$. Hence as a direct consequence of Theorem \ref{mainthm2}, we have

\begin{coro}
Let $(M^n, \bar{g})$ be as in Theorem \ref{mainthm2}. Assume that the radial sectional curvature of $N$ satisfies \eqref{decaycond1}. Let $\Sigma^{n-1}$ be a complete, non-compact star-shaped hypersurface  in $M$ which is contained in a slab $[r_1,r_2]\times N$ with $\sup_{\Sigma}\|h\|^2<+\infty$. Assume
 $-2E_{(k)}$ is positive semi-definite on $\Sigma$ and $\sup_\Sigma\mathcal{H}_{2k}<\infty$  on $\Sigma$.
If the quotient $\frac{\mathcal{H}_{2k+1}}{\mathcal{H}_{2k}}$ is constant, then the hypersurface is a slice $\{r_0\}\times N$.
\end{coro}

\vspace{2mm}
Following the argument close to the one of the proof of Theorem \ref{mainthm2}, one may  prove the Bernstein-type result in this case.

\vspace{2mm}

\noindent\textit{Proof of Theorem \ref{mainthm3}:}
By the generalized Omori-Yau maximum principle to the height function $r$, there exists a sequence $\{p_i\}\subset\Sigma$ such that
$$\lim_{i\rightarrow\infty} r(p_i)=\sup_{\Sigma} r,\;\lim_{i\rightarrow\infty}|\nabla r|(p_i)=0, \;\mbox{and} \lim_{i\rightarrow\infty} \sup \tr(-2E_{(k)}\nabla^2 r)(p_i)\leq 0.$$
It follows from the  semi-definity of $-2E_{(k)}$ and the positivity of $\mathcal{H}_{2k}$ that
$$0\leq \langle -2E_{(k)} \nabla r,\nabla r\rangle \leq \tr(-2E_{(k)})|\nabla r|^2\leq (n-1-2k)C_2|\nabla r|^2.$$
From the fact $\langle \partial_r,\nu\rangle^2=1-|\nabla r|^2$, we  have
$$\lim_{i\rightarrow\infty}\langle \partial_r,\nu\rangle (p_i)=1,$$
and thus
$$\lim_{i\rightarrow\infty}\langle -2E_{(k)}\nabla r,\nabla r\rangle(p_i) =0.$$

Combining all the above facts together into (\ref{eqq4}), we have
$$0\geq \lim_{i\rightarrow\infty}\sup \tr(-2E_{(k)}\nabla^2 r)(p_i)\geq C_1\lim_{i\rightarrow\infty}\left((n-1-2k)\frac{\l'(r)}{\l(r)}-\frac{\mathcal H_{2k+1}}{\mathcal H_{2k}}\right)(p_i)\geq 0, $$
so that
$$\lim_{i\rightarrow\infty}\left((n-1-2k)\frac{\l'(r)}{\l(r)}-\frac{\mathcal H_{2k+1}}{\mathcal H_{2k}}\right)(p_i)=0.$$
From the hypothesis, we have $\inf_{\Sigma}\left((n-1-2k)\frac{\l'(r)}{\l(r)}-\frac{\mathcal H_{2k+1}}{\mathcal H_{2k}}\right)=0,$
and thus $|\nabla r|\equiv0$ on $\Sigma$, which yields that $\Sigma$ is a slice $\{r_0\}\times M$ for some $r_0\in [0,\bar{r})$.
\qed

\

{\appendix

\section{Kottler-Schwarzschild manifolds}

The Kottler manifolds, or Kottler-Schwarzschild manifolds, are analogues of the Schwarzschild space
in the setting of asymptotically locally hyperbolic manifolds.  For $\kappa=1, 0\hbox{ or }-1$, let $(N(\kappa),\hat g)$ be a closed space form of constant sectional curvature $\kappa$. An $n$-dimensional Kottler-Schwarzschild manifold $P_{\kappa,m}= [\rho_{\kappa,m}, \infty)\times N(\kappa)$ is equipped with the metric
\begin{equation}\label{example_g}
g_{\kappa,m}=\frac{d{\rho}^2}{V_{\kappa,m}^2(\rho)}+{\rho}^2\hat g, \quad {V_{\kappa,m}}=\sqrt{{\rho}^2+\kappa-\frac{2m}{{\rho}^{n-2}}}.
\end{equation}
 Let $\rho_0:=\rho_{\kappa,m}$ be the largest positive root of
$$\phi(\rho):=\rho^2+\kappa-\frac{2m}{\rho^{n-2}}=0.$$
Remark that in (\ref{example_g}), in order to have a positive root $\rho_0$, if $\kappa= 0\hbox{ or }1$, the parameter $m$ should be always positive; if $\kappa=-1$, the parameter $m$ can be negative. In fact, in this case, $m$ belongs to the following interval
\begin{eqnarray*}\label{interval}
m\in[m_c,+\infty)\quad\mbox{and}\quad m_{c}=-\frac{(n-2)^{\frac{n-2}{2}}}{n^{\frac n2}}.
\end{eqnarray*}
Here the certain critical value $m_c$ comes from the following. If $m\leq 0$, one can solve the equation
$$\phi'(\rho)=2\rho+(n-2)\frac{2m}{\rho^{n-1}}=0,$$
to get the root $\rho_1=\left(-(n-2)m\right)^{\frac 1n}.$ Note the fact that $\phi(\rho_1)\leq 0$, which yields
$$m\geq-\frac{(n-2)^{\frac{n-2}{2}}}{n^{\frac n2}}.$$

By a change of variable $r=r(\rho)$ with $$r'(\rho)=\frac{1}{V_{\kappa,m}(\rho)},\quad r(\rho_{\kappa,m})=0,$$ we can rewrite $P_{\kappa,m}$ as a warped product manifold $P_{\kappa,m}=[0,\infty)\times_{\lambda_\kappa} N(\kappa)$  equipped with the metric
\begin{eqnarray*}
g_{\kappa,m}:=\bar g:=dr^2+\l_{\kappa}(r)^2\hat g,
\end{eqnarray*}
where $\l_{\kappa}: [0,\infty)\to [\rho_{\kappa,m},\infty)$ is the inverse of $r(\rho)$, i.e., $\l_{\kappa}(r(\rho))=\rho$.

It is easy to check
\begin{eqnarray*}
\l'_{\kappa}(r)&=&V_{\kappa,m}(\rho)=\sqrt{\kappa+\l_{\kappa}(r)^2-2m\l_{\kappa}(r)^{2-n}},\\
\l''_{\kappa}(r)&=&\l_{\kappa}(r)+(n-2)m\l_{\kappa}(r)^{1-n}.
\end{eqnarray*}
Hence
\begin{eqnarray*}
\l_{\kappa}\l''_{\kappa}-(\l'_{\kappa})^2=-\kappa+nm\lambda_\k^{2-n}.\end{eqnarray*}

For the case $\kappa=0$, $m\geq 0$ and hence $\l_{\kappa}\l''_{\kappa}-(\l'_{\kappa})^2=nm\lambda_\k^{2-n}\geq 0$.
For the case $\kappa=-1$, if $m\geq 0$, then $\l_{\kappa}\l''_{\kappa}-(\l'_{\kappa})^2=1+nm\lambda_\k^{2-n}>0$.
If $m\in [-\frac{(n-2)^{\frac{n-2}{2}}}{n^{\frac n2}},0)$, then
\begin{eqnarray*}
\l_{\kappa}\l''_{\kappa}-(\l'_{\kappa})^2&=&1+nm\lambda_\k^{2-n}\geq 1+nm\rho_0^{2-n}\geq 1+nm\rho_1^{2-n}\\&= &1+nm \left(-(n-2)m\right)^{\frac{2-n}{n}}=1-n(n-2)^{\frac{2-n}{n}}(-m)^\frac2n \\&\geq &1-n(n-2)^{\frac{2-n}{n}}\left(\frac{(n-2)^{\frac{n-2}{2}}}{n^{\frac n2}}\right)^\frac2n=0.\end{eqnarray*}

As a conclusion, the condition on the log convexity of $\l$  holds for the Kottler-Schwarzschild manifolds with $\kappa=0$ and $-1$. We remark that the log convexity of $\l$  does not hold for the Kottler-Schwarzschild manifolds with $\kappa=1$.
}

\

\noindent{\bf Acknowledgment.}
Both authors would like to thank  Prof. Guofang Wang for his encouragement and constant support.

\

\end{document}